\documentclass[11pt,a4papar]{article}

\usepackage{amsthm}
\usepackage{amssymb}
\usepackage{amscd}
\usepackage{amsmath}
\usepackage[all]{xy}
\usepackage[top=30truemm,bottom=30truemm,left=25truemm,right=25truemm]{geometry}
\usepackage{comment}
\usepackage{algorithmic,algorithm}
\usepackage{here}
\usepackage{graphicx}
\usepackage{color} %
\usepackage{enumerate} %
\usepackage{listliketab}
\usepackage{multirow}
\usepackage{cases}


\makeatletter
  
  \@addtoreset{algorithm}{subsection}
\makeatother
\theoremstyle{definition}

\theoremstyle{definition}

\theoremstyle{plain}
\newtheorem{theo}{Theorem}%[section]

\theoremstyle{plain}
\newtheorem{coro}{Corollary}%[section]
\theoremstyle{plain}
\theoremstyle{plain}

\theoremstyle{plain}

\theoremstyle{plain}

\theoremstyle{plain}
\newtheorem{thm}{Theorem}[section]%[subsection]
\theoremstyle{definition}

\theoremstyle{definition}

\theoremstyle{definition}

\theoremstyle{definition}

\theoremstyle{definition}
\newtheorem{rem}[thm]{Remark}
\theoremstyle{plain}
\newtheorem{prop}[thm]{Proposition}
\theoremstyle{plain}
\newtheorem{lem}[thm]{Lemma}
\theoremstyle{plain}
\newtheorem{cor}[thm]{Corollary}
\theoremstyle{definition}

\theoremstyle{definition}

\theoremstyle{plain}
\newtheorem{conje}[thm]{Conjecture}
\theoremstyle{definition}

\theoremstyle{plain}
\newtheorem{prob}[thm]{Problem}
\numberwithin{equation}{section}

\def\F{\mathbb{F}}

\begin{document}

%Title
\title{On the existence of superspecial nonhyperelliptic curves of genus $4$}
\author{Momonari Kudo\thanks{Kobe City College of Technology.}
\thanks{Institute of Mathematics for Industry, Kyushu University. E-mail: \texttt{m-kudo@math.kyushu-u.ac.jp}}}
\providecommand{\keywords}[1]{\textbf{\textit{Key words---}} #1}
\maketitle

\begin{abstract}
A curve over a perfect field $K$ of characteristic $p > 0$ is said to be {\it superspecial} if its Jacobian is isomorphic to a product of supersingular elliptic curves over the algebraic closure $\overline{K}$.
In recent years, isomorphism classes of superspecial nonhyperelliptic curves of genus $4$ over finite fields in small characteristic have been enumerated.
In particular, the non-existence of superspecial curves of genus $4$ in characteristic $p = 7$ was proved.
In this note, we give an elementary proof of the existence of superspecial nonhyperelliptic curves of genus $4$ for infinitely many primes $p$.
Specifically, we prove that the variety $C_p : x^3+y^3+w^3= 2 y w + z^2 = 0$ in the projective $3$-space with $p > 2$ is a superspecial curve of genus $4$ if and only if $p \equiv 2 \pmod{3}$.
Our computational results show that $C_p$ with $p \equiv 2 \pmod 3$ are {\it maximal} curves over $\mathbb{F}_{p^2}$ for all $3 \leq p \leq 269$.
\end{abstract}


%=========================
\section{Introduction}\label{sec:intro}
%=========================

Let $p$ be a rational prime greater than $2$, and let $\mathbb{F}_q$ denote the finite field of $q$ elements, where $q$ is a power of prime.
Let $K$ be an arbitrary perfect field of characteristic $p$.
We denote by $\overline{K}$ the algebraic closure of $K$.
By a curve, we mean a non-singular projective variety of dimension one.
Let $C$ be a curve of genus $g$ over $K$.
We say that $C$ is {\it superspecial} if its Jacobian is isomorphic to the product of $g$ supersingular elliptic curves over $\overline{K}$.
The existence of a superspecial curve over an algebraically closed field in characteristic $p$ implies that there exists a maximal or minimal curve over $\mathbb{F}_{p^2}$.
Here a curve over $\mathbb{F}_q$ is called a maximal (resp.\ minimal) curve if the number of its $\mathbb{F}_q$-rational points attains the Hasse-Weil upper (resp.\ lower) bound $q + 1 + 2 g \sqrt{q}$ (resp.\ $q + 1 - 2 g \sqrt{q}$).
Conversely, any maximal or minimal curve over $\mathbb{F}_{p^2}$ is superspecial.
This work aims to find a lot of superspecial curves and maximal curves for a given genus.
Note that for a fixed pair $(g,q)$, superspecial curves over $\overline{\mathbb{F}_q}$ of genus $g$ are very rare:
the number of such curves is finite, whereas the whole set of curves over $\overline{\mathbb{F}_q}$ of genus $g$ has dimension $3 g - 3$.
% Superspecial curves over $\mathbb{F}_q$ of higher genus $g$ are more rarer that those of lower $g$.
Thus, finding superspecial curves over $\mathbb{F}_q$ of higher genus $g$ is more difficult than finding those of lower genus $g$.
% This work aims to find such a curve for a given genus.

In the case of $g\le 3$ and in the case of hyperelliptic curves, many results on the existence and enumeration of superspecial/maximal curves are known, see e.g., \cite{Deuring}, \cite[Prop.\ 4.4]{XYY16} for $g=1$, \cite{HI}, \cite{IK}, \cite{Serre1983} for $g=2$, \cite{Hashimoto}, \cite{Ibukiyama} for $g=3$, and \cite{Taf}, \cite{Taf2} for hyperelliptic curves.
In particular, it is well-known that there exist supersingular (and thus superspecial) elliptic curves in characteristic $p$ for infinitely many primes $p$ (see, e.g., \cite[Examples 4.4 and 4.5]{Silv})).
For example, the elliptic curve $E_p : y^2 = x^3 + 1$ with $p \geq 5$ is supersingular if and only if $p \equiv 2 \pmod{3}$.
Moreover, the set of primes $p$ for which $E_p$ is supersingular has natural density $1/2$.

In the case of {\it nonhyperelliptic} curves of genus $g=4$, Fuhrmann-Garcia-Torres proved in \cite{FGT} that there exists a maximal (and superspecial) curve $C_{0}$ of $g=4$ over $K= \F_{5^2}$, and that it gives a unique $\overline{K}$-isomorphism class.
In \cite{KH16}, \cite{KH17-2} and \cite{KH17a}, the isomorphism classes of superspecial nonhyperelliptic curves of genus $4$ over finite fields are enumerated in characteristic $p \leq 11$.
% Note that all the maximal curves over $K=\mathbb{F}_{5^2}$ enumerated in \cite{KH16} are included in the unique isomorphism class of $C_0$ over $\overline{K}$.
Results in \cite{KH16}, \cite{KH17-2} and \cite{KH17a} also show that there exist superspecial nonhyperelliptic curves of genus $4$ in characteristic $5$ and $11$, whereas there does not exist such a curve in characteristic $7$.

The objective of this note is to investigate whether a superspecial nonhyperelliptic curve of genus $g=4$ exists or not for $p \geq 13$.
In contrast to the rarity of superspecial curves of higher genus, our main results (Theorem \ref{thm:main} and Corollary \ref{cor:main} below) show the existence of superspecial curves of genus $g=4$ in characteristic $p$ for half of the primes as well as the case of $g=1$.
%Here is our main result:

\begin{theo}%\label{thm:main0}
Put $Q:= 2 y w + z^2$ and $P := x^3+y^3+w^3$.
Let $C_p = V(Q, P)$ denote the projective zero-locus in $\mathbf{P}^3 = \mathrm{Proj}(\overline{K}[x,y,z,w])$ defined by $Q=0$ and $P=0$.
Then $C_p$ is a superspecial nonhyperelliptic  curve of genus $4$ if and only if $p \equiv 2 \pmod{3}$.
\end{theo}

We prove Theorem \ref{thm:main} by simple computations in linear and fundamental commutative algebra and in combinatorics together with results in \cite{KH16}, \cite{KH17-2} and \cite{KH17a} (so this note also complements results in these three previous papers).
As a corollary of this theorem, we have the following:

\begin{coro}%\label{cor:main0}
There exist superspecial nonhyperelliptic curves of genus $4$ in characteristic $p$ for infinitely many primes $p$.
The set of primes $p$ for which $C_p$ is superspecial has natural density $1/2$. 
\end{coro}

%the density of the set of primes for which our curves are superspecial is $1/2$.

Theorem \ref{thm:main} and Corollary \ref{cor:main} also give a partial answer to the genus $4$ case of the problem proposed by Ekedahl in 1987, see p.\ 173 of \cite{Ekedahl}.
In Section \ref{sec:comp}, we give a table of the number of $\mathbb{F}_{p^2}$-rational points on $C_p$ for $3 \leq p \leq 269$ obtained by using a computer algebra system Magma~\cite{Magma}.
As computational results, we found maximal nonhyperelliptic curves of genus $4$ over $\mathbb{F}_{p^2}$.
Specifically, we have that for all $3 \leq p \leq 269$ with $p \equiv 2 \pmod 3$, the curves $C_p$ are maximal over $\mathbb{F}_{p^2}$.
% The following is our conjecture on the existence of $\mathbb{F}_{p^2}$-maximal nonhyperelliptic curves of genus $4$.

% \begin{conjec}
% For any $p$ with $p \equiv 2 \pmod{3}$, the curve $C_p$ over $\mathbb{F}_{p^2}$ is maximal.
% \end{conjec}

\subsection*{Acknowledgments}
The author thanks Shushi Harashita for his comments to the preliminary version of this note.
He gave the author information on the existence of superspecial curves of genus $g$ over $\mathbb{F}_q$ in the case of $g \leq 3$, in the case of $(g,q)=(4,13)$, and in the hyperelliptic case.
He also pointed out that computing the rational points of our curves is reduced into solving a diagonal equation.
%to the preliminary version of this note.
% The first author thanks Kazuhiro Yokoyama for helpful advice on the Gr\"{o}bner basis computation, proofs with computer algebra systems and their correctnesses.
% The second author would like to thank Keiichi Gunji for his comments to the preliminary version of this paper.
% The authors are very grateful to
% Gerard van der Geer and Everett W. Howe,
% who gave us helpful comments and suggestions
% after we uploaded the first version of this paper to arXiv.
% This work was supported by
% JSPS Grant-in-Aid for Young Scientists (B) 21740006.

%=======================================================
\section{Superspecialty of curves $x^3+y^3+w^3= 2 y w + z^2 = 0$}\label{sec:pre}
%=======================================================

As in the previous section, let $K$ be a perfect field of characteristic $p>2$.
Let $K[x,y,z,w]$ denote the polynomial ring of the four variables $x$, $y$, $z$ and $w$ over $K$.
% In \cite{KH16}, \cite{KH17-2} and \cite{KH17a}, the isomorphism classes of superspecial curves of genus $4$ over finite fields in characteristic $p \leq 11$ are enumerated.
% In particular, there exist superspecial nonhyperelliptic curves of genus $4$ in characteristic $5$ and $11$, whereas there does not exist such a curve in characteristic $7$.
As examples of superspecial curves of genus $g=4$ in characteristic $p=5$ and $11$, we have the projective varieties in the projective $3$-space $\mathbf{P}^3 = \mathrm{Proj}(\overline{K}[x,y,z,w])$ defined by the same systems of equations: $x^3+y^3+w^3=0$ and $2 y w + z^2 = 0$, see \cite[Exmaple 6.2.1]{KH16} and \cite[Proposition 4.4.4]{KH17-2}.
% A goal of this paper is to show that the variety $x^3+y^3+w^3= 2 y w + z^2 = 0$ over $K$ is a superspecial curve of genus $4$ if and only if $p \equiv 2 \pmod{3}$.

In this section, we shall prove that the variety $x^3+y^3+w^3= 2 y w + z^2 = 0$ over $K$ is (resp.\ not) a superspecial curve of genus $4$ if $p \equiv 2 \pmod{3}$ (resp.\ $p \equiv 1 \pmod{3}$).
Throughout this section, we set $Q:= 2 y w + z^2$ and $P := x^3+y^3+w^3$.
Let $C_p$ denote the projective variety $V(Q,P)$ in $\mathbf{P}^3$ defined by $P= Q = 0$ in characteristic $p$.
First, we prove that the variety $C_p$ is non-singular (resp.\ singular) if $p > 3$ (resp.\ $p=3$).

\begin{lem}\label{lem:sing}
If $p > 3$ $($resp.\ $p=3)$, then the variety $C_p = V(Q,P)$ is non-singular $($resp.\ singular$)$.
\end{lem}

\begin{proof}
Let $J (P, Q)$ denote the set of all the minors of degree $2$ of the Jacobian matrix
\[
\left( \begin{array}{cccc}
	\frac{\partial P}{\partial x} & \frac{\partial P}{\partial y} & \frac{\partial P}{\partial z} & \frac{\partial P}{\partial w} \\
	\frac{\partial Q}{\partial x} & \frac{\partial Q}{\partial y} & \frac{\partial Q}{\partial z} & \frac{\partial Q}{\partial w} 
	\end{array} \right)
	=
	\left( \begin{array}{cccc}
	3 x^2 & 3 y^2 & 0 & 3 w^2 \\
	0 & 2 w & 2 z & 2 y
	\end{array} \right) .
\]
Namely, the set $J (P, Q)$ consists of the following $6$ elements:
\begin{eqnarray}
f_1 & := & \frac{\partial P}{\partial x} \cdot \frac{\partial Q}{\partial y} -  \frac{\partial P}{\partial y} \cdot \frac{\partial Q}{\partial x} = 6 x^2 w, \nonumber \\
f_2 & := & \frac{\partial P}{\partial x} \cdot \frac{\partial Q}{\partial z} -  \frac{\partial P}{\partial z} \cdot \frac{\partial Q}{\partial x} = 6 x^2 z, \nonumber \\
f_3 & := & \frac{\partial P}{\partial x} \cdot \frac{\partial Q}{\partial w} -  \frac{\partial P}{\partial w} \cdot \frac{\partial Q}{\partial x} = 6 x^2 y, \nonumber \\
f_4 & := & \frac{\partial P}{\partial y} \cdot \frac{\partial Q}{\partial z} -  \frac{\partial P}{\partial z} \cdot \frac{\partial Q}{\partial y} = 6 y^2 z, \nonumber \\
f_5 & := & \frac{\partial P}{\partial y} \cdot \frac{\partial Q}{\partial w} -  \frac{\partial P}{\partial w} \cdot \frac{\partial Q}{\partial y} = 6 y^3 - 6 w^3, \nonumber \\
f_6 & := & \frac{\partial P}{\partial z} \cdot \frac{\partial Q}{\partial w} -  \frac{\partial P}{\partial w} \cdot \frac{\partial Q}{\partial z} = - 6 z w^2. \nonumber 
\end{eqnarray}
Assume $p > 3$.
It suffices to show that $x$, $y$, $z$ and $w$ belong to the radical of the ideal generated by $P$, $Q$ and $J (P, Q)$.
By straightforward computations, we have
\begin{eqnarray}
x^2 P - (6^{-1} y^2) f_3 - (6^{-1} w^2) f_1 & = & x^5, \nonumber \\
y P - (6^{-1} x) f_3 - (6^{-1} y) f_5 & = & 2 y^4, \nonumber \\
(-2 y z w + z^3) Q + (2 \cdot 3^{-1} w^2 ) f_4 & = & z^5, \nonumber \\
w P - (6^{-1} x) f_1 - (6^{-1} w) f_5 & = & 2 w^4, \nonumber
\end{eqnarray}
which belong to the ideal $\langle P, Q, J (P,Q) \rangle$ in $K[x,y,z,w]$.
Thus, $x$, $y$, $z$ and $w$ belong to its radical.

If $p=3$, then $J ( P, Q ) = \{ 0 \}$, and hence all the points on $V(Q,P)$ are singular points.
\end{proof}

In the following, we suppose $p > 3$.
It is shown in \cite{KH16} that we can decide whether $C_p$ is superspecial or not by computing the coefficients of certain monomials in $(Q P)^{p-1}$.

\begin{prop}[\cite{KH16}, Corollary 3.1.6]\label{cor:HW}
With notation as above, the curve $C_p$ is superspecial if and only if the coefficients of all the following $16$ monomials of degree $5 (p-1)$ in $(Q P)^{p-1}$ are zero:
\begin{equation}
\begin{array}{cccc}
( x^2 y z w )^{p-1}, & x^{2 p-1} y^{p-2} z^{p-1} w^{p-1}, & x^{2 p-1} y^{p-1} z^{p - 2} w^{p -1}, &  x^{2 p -1} y^{p-1} z^{p - 1} w^{p -2}, \\
x^{p-2} y^{2 p-1} z^{p-1} w^{p-1}, & ( x y^2 z w )^{p-1} , & x^{p-1} y^{2 p-1} z^{p - 2} w^{p -1}, &  x^{p -1} y^{2 p-1} z^{p - 1} w^{p -2}, \\
x^{p-2} y^{p-1} z^{2 p - 1} w^{p -1}, & x^{p-1} y^{p-2} z^{2 p-1} w^{p-1}, & ( x y z^2 w )^{p-1} , &  x^{p -1} y^{p-1} z^{2 p - 1} w^{p -2}, \\
 x^{p -2} y^{p-1} z^{p - 1} w^{2 p -1}, & x^{p-1} y^{p-2} z^{p-1} w^{2 p-1}, & x^{p-1} y^{p-1} z^{p - 2} w^{2 p -1}, & ( x y z w^2 )^{p-1} .
\end{array} \nonumber
\end{equation}
\end{prop}

To prove Theorem \ref{thm:main} stated in Section \ref{sec:intro} (and in Section \ref{sec:main}), we compute the $16$ coefficients given in Proposition \ref{cor:HW}.
Note that we have $Q P = x^3 z^2 + y^3 z^2 + 2 x^3 y w + 2 y^4 w + z^2 w^3 + 2 y w^4$, and
\begin{eqnarray}
(Q P)^{p-1} & = & 
\sum_{a + b + c + d + e + f = p-1} \binom{p-1}{a,b,c,d,e,f} ( x^3 z^2 )^{a} (y^3 z^2)^b (2 x^3 y w)^c  (2 y^4 w)^d (z^2 w^3)^e (2 y w^4)^f \nonumber \\
& = & \sum_{a + b + c + d + e + f = p-1} \binom{p-1}{a,b,c,d,e,f} ( x^{3 a} z^{2 a} ) (y^{3 b} z^{2 b}) (2^c x^{3 c} y^c w^c)  (2^d y^{4 d} w^d) (z^{2 e} w^{3 e} ) (2^f y^f w^{4 f}) \nonumber \\
& = & \sum_{a + b + c + d + e + f = p-1} 2^{c+d+f} \cdot \binom{p-1}{a,b,c,d,e,f} x^{3 a + 3 c} y^{3 b + c + 4 d + f} z^{2 a + 2 b + 2 e} w^{c + d+ 3 e + 4 f} \label{eq:QP}
\end{eqnarray}
by the multinomial theorem.
To express $(QP)^{p-1}$ as a sum of the form
\[
(Q P)^{p-1} = \sum_{(i,j,k,\ell) \in \left( \mathbb{Z}_{\geq 0} \right)^{\oplus 4}} c_{i,j,k,\ell} x^i y^j z^k w^{\ell},
\]
we consider the linear system
\begin{eqnarray}
  \left\{
    \begin{array}{l}
     a + b + c + d + e + f = p-1, \\
     3 a + 3 c  = i, \\
     3 b + c + 4 d + f = j,  \\
     2 a + 2 b + 2 e = k,  \\
     c + d + 3 e + 4 f = \ell , 
    \end{array}
  \right. \label{eq:system}
\end{eqnarray}
and put
\begin{eqnarray}
S (i,j,k,\ell) := \{ (a, b, c, d, e, f) \in [0,p-1]^{\oplus 6} : (a,b,c,d,e,f) \mbox{ satisfies } \eqref{eq:system} \} \label{eq:sol_set}
\end{eqnarray}
for each $(i,j,k,\ell) \in \left( \mathbb{Z}_{\geq 0} \right)^{\oplus 4}$.
Using the notation $S (i,j,k,\ell)$, we have
\begin{eqnarray}
(Q P)^{p-1} = \sum_{(i,j,k,\ell) \in \left( \mathbb{Z}_{\geq 0} \right)^{\oplus 4}} \left( \sum_{(a, b, c, d, e, f) \in S(i,j,k,\ell )} 2^{c+d+f} \cdot \binom{p-1}{a,b,c,d,e,f} \right) x^{i} y^{j} z^{k} w^{\ell }. \label{eq:sum}
\end{eqnarray}

\begin{lem}\label{lem:coeff_zero}
With notation as above, the coefficients of the monomials $x^i y^j z^{p-2} w^{\ell}$ and $x^i y^j z^{2 p-1} w^{\ell}$ in $(Q P)^{p-1}$ are zero for all $(i,j,\ell) \in \left( \mathbb{Z}_{\geq 0} \right)^{\oplus 3}$.
\end{lem}

\begin{proof}
Recall from \eqref{eq:QP} that the $z$-exponent of each monomial in $(Q P)^{p-1}$ is $2 a + 2 b + 2 e$, which is an even number.
On the other hand, the $z$-exponents of the monomials $x^i y^j z^{p-2} w^{\ell}$ and $x^i y^j z^{2 p-1} w^{\ell}$ are odd numbers, and thus their coefficients in $(Q P)^{p-1}$ are all zero.
\end{proof}

% \begin{prop}[\cite{KH16}, Corollary 3.1.6]\label{cor:HW}
% With notation as above, $C$ is superspecial if and only if the coefficients of $x^{pi-i'}y^{pj-j'}z^{pk-k'}w^{pl-l'}$
% in $(P Q)^{p-1}$ are equal to $0$
% for all positive integers $i,j,k,l,i',j',k',l'$ with $i+j+k+l=i'+j'+k'+l'=5$.
% \end{prop}

Let $\mathcal{M}$ be the set of the $16$ monomials given in Proposition \ref{cor:HW}, and set
\[
E (\mathcal{M}) := \{ (i,j,k,\ell ) \in \left( \mathbb{Z}_{\geq 0} \right)^{\oplus 4} : x^i y^j z^k w^{\ell} = m \mbox{ for some } m \in \mathcal{M} \} ,
\]
which is the set of the exponent vectors of the monomials in $\mathcal{M}$.

\begin{lem}\label{lem:coeff-1}
Assume $p \equiv 2 \pmod{3}$.
Then we have $S (i,j,k, \ell) = \emptyset$ for any $(i,j,k,\ell) \in E(\mathcal{M})$.
\end{lem}

\begin{proof}
Note that for each $(i,j,k,\ell) \in E(\mathcal{M})$, we have $i+j+k+\ell=5 (p-1)$, see Proposition \ref{cor:HW}.
Using matrices, we write the system \eqref{eq:system} as
\begin{eqnarray}
	\left( \begin{array}{cccccc}
	1 & 1 & 1 & 1 & 1 & 1 \\
	3 & 0 & 3 & 0 & 0 & 0 \\
	0 & 3 & 1 & 4 & 0 & 1 \\
	2 & 2 & 0 & 0 & 2 & 0 \\
	0 & 0 & 1 & 1 & 3 & 4
	\end{array} \right)
	\left( \begin{array}{c}
	a \\
	b \\
	c \\
	d \\
	e \\
	f 
	\end{array} \right)
	= \left( \begin{array}{c}
	p-1 \\
	i \\
	j \\
	k \\
	\ell \\
	\end{array} \right) , \label{eq:sys1-1}
\end{eqnarray}
% \[
% 	\left( \begin{array}{cccccc}
% 	1 & 1 & 1 & 1 & 1 & 1 \\
% 	0 & 0 & 0 & 0 & 0 & 0 \\
% 	0 & 0 & 1 & 1 & 0 & 1 \\
% 	2 & 2 & 0 & 0 & 2 & 0 \\
% 	0 & 0 & 1 & 1 & 0 & 1
% 	\end{array} \right)
% 	\left( \begin{array}{c}
% 	a \\
% 	b \\
% 	c \\
% 	d \\
% 	e \\
% 	f 
% 	\end{array} \right)
% 	= \left( \begin{array}{c}
% 	p-1 \\
% 	i \\
% 	j \\
% 	k \\
% 	\ell \\
% 	\end{array} \right)
% \]
whose extended coefficient matrix is transformed as follows:
\if 0
\[
	\left( \begin{array}{ccccccc}
	1 & 1 & 1 & 1 & 1 & 1 & p-1\\
	0 & 0 & 0 & 0 & 0 & 0 & i\\
	0 & 0 & 1 & 1 & 0 & 1 & j \\
	2 & 2 & 0 & 0 & 2 & 0 & k \\
	0 & 0 & 1 & 1 & 0 & 1 & \ell
	\end{array} \right)
% 	\longrightarrow
% 	\left( \begin{array}{ccccccc}
% 	1 & 1 & 1 & 1 & 1 & 1 & p-1\\
% 	2 & 2 & 0 & 0 & 2 & 0 & k \\
% 	0 & 0 & 1 & 1 & 0 & 1 & j \\
% 	0 & 0 & 1 & 1 & 0 & 1 & \ell \\
% 	0 & 0 & 0 & 0 & 0 & 0 & i
% 	\end{array} \right)
	\longrightarrow
	\left( \begin{array}{ccccccc}
	1 & 1 & 1 & 1 & 1 & 1 & p-1\\
	0 & 0 & 1 & 1 & 0 & 1 & j \\
	0 & 0 & 0 & 0 & 0 & 0 & k + p - 1 - j \\
	0 & 0 & 0 & 0 & 0 & 0 & \ell - j\\
	0 & 0 & 0 & 0 & 0 & 0 & i
	\end{array} \right)
\]
\fi
\[
\left( \begin{array}{ccccccc}
	1 & 1 & 1 & 1 & 1 & 1 & p-1 \\
	3 & 0 & 3 & 0 & 0 & 0 & i\\
	0 & 3 & 1 & 4 & 0 & 1 & j\\
	2 & 2 & 0 & 0 & 2 & 0 & k\\
	0 & 0 & 1 & 1 & 3 & 4 & \ell
	\end{array} \right)
\longrightarrow
	\left( \begin{array}{ccccccc}
	1 & 1 & 1 & 1 & 1 & 1 & p-1\\
	0 & 3 & 1 & 4 & 0 & 1 & j \\
	0 & 0 & 1 & 1 & -3 & -2 & i + j- 3 (p - 1)  \\
	0 & 0 & 0 & 0 & 6 & 6 & \ell - (i + j - 3 (p-1) )\\
	0 & 0 & 0 & 0 & 0 & 0 & 0
	\end{array} \right)
\]
Considering modulo $3$, we have the following linear system over $\mathbb{F}_3$:
\[
\left( \begin{array}{cccccc}
	1 & 1 & 1 & 1 & 1 & 1 \\
	0 & 0 & 1 & 1 & 0 & 1  \\
	0 & 0 & 1 & 1 & 0 & 1  \\
	0 & 0 & 0 & 0 & 0 & 0 \\
	0 & 0 & 0 & 0 & 0 & 0 
	\end{array} \right)
	\left( \begin{array}{c}
	a^{\prime} \\
	b^{\prime} \\
	c^{\prime} \\
	d^{\prime} \\
	e^{\prime} \\
	f^{\prime} 
	\end{array} \right)
	= \left( \begin{array}{c}
	p-1 \\
	j \\
	i+j \\
    \ell- (i+j)  \\
	0 \\
	\end{array} \right) ,
\]
which is equivalent to
\begin{eqnarray}
\left( \begin{array}{cccccc}
	1 & 1 & 1 & 1 & 1 & 1 \\
	0 & 0 & 1 & 1 & 0 & 1  \\
	0 & 0 & 0 & 0 & 0 & 0  \\
	0 & 0 & 0 & 0 & 0 & 0 \\
	0 & 0 & 0 & 0 & 0 & 0 
	\end{array} \right)
	\left( \begin{array}{c}
	a^{\prime} \\
	b^{\prime} \\
	c^{\prime} \\
	d^{\prime} \\
	e^{\prime} \\
	f^{\prime} 
	\end{array} \right)
	= \left( \begin{array}{c}
	p-1 \\
	j \\
	i \\
    \ell- (i+j)  \\
	0 \\
	\end{array} \right) . \label{eq:sys2}
\end{eqnarray}
Note that the system \eqref{eq:sys2} over $\mathbb{F}_3$ has a solution if and only if $i \equiv 0 \pmod{3}$ and $\ell \equiv j \pmod{3}$.
We claim that if $p \equiv 2 \pmod{3}$, the original system \eqref{eq:sys1-1} over $\mathbb{Z}$ has no solution in $[0,p-1]^{\oplus 6}$ for any $(i,j,k,\ell) \in E (\mathcal{M})$.
Indeed, if $p \equiv 2 \pmod{3}$ and if the system \eqref{eq:sys1-1} has a solution in $[0,p-1]^{\oplus 6}$ for some $(i,j,k,\ell) \in E (\mathcal{M})$, the system \eqref{eq:sys2} has a solution.
By Lemma \ref{lem:coeff_zero}, we may assume $k \neq p-2$ and $k \neq 2 p-1$, i.e., $k = 2 p -2$ or $k= p-1$.
Since $i \equiv 0 \pmod{3}$ and since $p \equiv 2 \pmod{3}$, the integer $i$ is equal to $2 p -1$ or $p-2$, and thus $(i,j,k,\ell) = (2 p -1, p-2, p-1, p-1)$, $(2p-1,p-1,p-1,p-2)$, $(p -2, 2p-1, p-1, p-1)$ or $(p-2,p-1,p-1,2p-1)$.
However, any of the above four candidates for $(i,j,k,\ell)$ does not satisfy $\ell \equiv j \pmod{3}$, which is a contradiction.%\qed
% It follows from $\ell \equiv j \pmod{3}$ that $(i,j,k,\ell) = (p -2, p-1, 2 p-1, p-1)$, but there is no $(a, b, e)$ with $2 (a + b + e) = 2 p-1$, which contradicts the existence of a solution of the original system.
\end{proof}

\begin{prop}\label{prop:main1}
Assume $p \equiv 2 \pmod{3}$.
Then the curve $C_p = V(Q,P)$ is superspecial.
\end{prop}
\begin{proof}
It follows from Lemma \ref{lem:coeff-1} that the coefficient of $x^{i} y^{j} z^{k} w^{\ell}$ in \eqref{eq:sum} is zero for each $(i,j,k,\ell) \in E(\mathcal{M})$. 
By Proposition \ref{cor:HW}, the curve $V(Q,P)$ is superspecial.%\qed
\end{proof}

It follows from the proof of Lemma \ref{lem:coeff-1} that \eqref{eq:system} is equivalent to the following system:
\begin{eqnarray}
  \left\{
    \begin{array}{l}
     a + b + c + d + e + f = p-1, \\
     3 b + c + 4 d + f  = j,  \\
     c + d - 3 e - 2 f = i + j - 3 (p-1),  \\
     6 e + 6 f = \ell - (i + j - 3 (p-1) ).  
    \end{array}
  \right. \label{eq:system3}
\end{eqnarray}

Next, we consider the case of $p \equiv 1 \pmod{3}$.

\begin{lem}\label{lem:coeff}
Assume $p \equiv 1 \pmod{3}$.
Then we have $\# S (p-1, p-1, 2 p -2 , p-1) = 1$.
In other words, the system \eqref{eq:system3} with $(i,j,k,\ell) = (p-1,p-1,2 p - 2, p-1)$ has a unique solution in $[0,p-1]^{\oplus 6}$.
The solution is given by
\begin{eqnarray}
\left( \begin{array}{cccccc}
a, & b, & c, & d, & e, & f
\end{array} \right)
 = 
\left( \begin{array}{cccccc}
 (p-1)/3, & (p-1)/3, & 0, &  0, & (p-1)/3, & 0
 \end{array} \right) . \label{eq:sol}
\end{eqnarray}
\end{lem}

\begin{proof}
The system to be solved with $(i,j,k,\ell) = (p-1,p-1,2 p - 2, p-1)$ is given by
\begin{numcases}
  {}
  a + b + c + d + e + f = p-1, & \label{eq:1}\\
  3 b + c + 4 d + f  = p-1, & \label{eq:2}\\
  c + d - 3 e - 2 f =  - (p-1), & \label{eq:3} \\
  6 e + 6 f = 2(p-1) & \label{eq:4} 
\end{numcases}
with $(a, b, c,d, e, f) \in [0,p-1]^{\oplus 6}$.
Since $c + d - 3 e - 2 f = c + d + f - (3 e + 3 f)$, it follows from \eqref{eq:3} and \eqref{eq:4} that $c + d + f = 0$, and thus $c=d=f=0$.
By \eqref{eq:2} and \eqref{eq:4}, we have $b = e = (p-1)/3$.
From \eqref{eq:1}, we also have $a = (p-1)/3$. %\qed
\end{proof}

\begin{lem}\label{lem:coeff2}
Assume $p \equiv 1 \pmod{3}$.
Then the coefficient of the monomial $x^{p-1} y^{p-1} z^{2p-2} w^{p-1}$ in $(Q P)^{p-1}$ is not zero.
\end{lem}

\begin{proof}
Let $c_{p-1,p-1,2p-2,p-1}$ be the coefficient of $x^{p-1} y^{p-1} z^{2p-2} w^{p-1}$ in $(Q P)^{p-1}$.
Recall from \eqref{eq:sum} that $c_{p-1,p-1,2p-2,p-1}$ is given by
\[
\sum_{(a,b,c,d,e,f) \in S(p-1,p-1,2p-2,p-1)} 2^{c+d+f} \cdot \binom{p-1}{a,b,c,d,e,f} ,
\]
where $S (p-1,p-1,2p-2,p-1)$ is defined in \eqref{eq:sol_set}.
By Lemma \ref{lem:coeff}, the set $S(p-1,p-1,2p-2,p-1)$ consists of only the element given by \eqref{eq:sol}, and hence
\[
c_{p-1,p-1,2 p-2,p-1} = \cfrac{(p-1)!}{\left( \cfrac{p-1}{3} \right) ! \left( \cfrac{p-1}{3} \right) ! \left( \cfrac{p-1}{3} \right) !},
\]
which is not divisible by $p$.
\end{proof}

% Assume that $p = 1 \pmod{3}$ and that the system has a solution for $(i,j,k,\ell)$.
% In this case, the exponent $i$ is $2 p -2$ or $p-1$, and thus $(i,j,k,\ell) = (2 p -2, p-1, p-1, p-1)$, $(p-1,2p-2,p-1,p-1)$, $(p -1, 2p-1, p-1, p-2)$, $(p-1,p-1,2p-2,p-1)$, $(p-1,p-1,p-1,2p-2)$ or $(p-1,p-2,p-1,2p-1)$.
% Since $\ell=j \pmod{3}$, we have $(i,j,k,\ell) = (2 p -2, p-1, p-1, p-1)$, $(p-1,2p-2,p-1,p-1)$, $(p-1,p-1,2p-2,p-1)$, or $(p-1,p-1,p-1,2p-2)$.

\begin{prop}\label{prop:main2}
Assume $p \equiv 1 \pmod{3}$.
Then the curve $C_p = V(Q,P)$ is not superspecial.
\end{prop}

\begin{proof}
It follows from Lemma \ref{lem:coeff2} that the coefficient of $x^{p-1} y^{p-1} z^{2p-2} w^{p-1}$ in $(Q P)^{p-1}$ is not zero.
By Proposition \ref{cor:HW}, the curve $V(Q,P)$ is not superspecial.
\end{proof}

%============================================================
\section{Proofs of main results and some further problems}\label{sec:main}
%============================================================
As in the previous section, let $K$ be a perfect field of characteristic $p>2$.
Here, we re-state Theorem \ref{thm:main} and Corollary \ref{cor:main} in Section \ref{sec:intro} and prove them:

\begin{thm}\label{thm:main}
Put $Q:= 2 y w + z^2$ and $P := x^3+y^3+w^3$.
Let $C_p = V(Q, P)$ denote the projective zero-locus in $\mathbf{P}^3 = \mathrm{Proj}(\overline{K}[x,y,z,w])$ defined by $Q=0$ and $P=0$.
Then $C_p$ is a superspecial nonhyperelliptic curve of genus $4$ if and only if $p \equiv 2 \pmod{3}$.
\end{thm}

\begin{proof}
Recall from Lemma \ref{lem:sing} that $C_p$ is singular if $p=3$, and non-singular if $p>3$.
We may assume $p>3$.
Since $C_p$ is the set of the zeros of the quadratic form $Q$ and the cubic form $P$ over $K$, it is a nonhyperelliptic curve of genus $4$ over $K$, see \cite[Section 2]{KH16}.
It follows from Propositions \ref{prop:main1} and \ref{prop:main2} that the non-singular curve $C_p$ is superspecial if and only if $p \equiv 2 \pmod{3}$.
\end{proof}

\begin{cor}\label{cor:main}
There exist superspecial nonhyperelliptic curves of genus $4$ in characteristic $p$ for infinitely many primes $p$.
The set of primes $p$ for which $C_p$ is superspecial has natural density $1/2$. 
\end{cor}

\begin{proof}
The first claim immediately follows from Theorem \ref{thm:main} and Dirichlet's Theorem.
The second claim is deduced from the fact that the natural density of primes equal to $2$ modulo $3$ is $1/\varphi (3) = 1/2$, where $\varphi$ is Euler's totient function.%\qed 
\end{proof}

\begin{prob}
Does there exist a superspecial curve of genus $4$ in characteristic $p$ for any $p > 13$ with $p \equiv 1 \pmod{3}$?
Cf.\ the non-existence for $p=7$ is already shown in {\rm \cite{KH16}}, whereas the existence for $p = 13$ is shown, see e.g., {\rm \cite{ManyPoints}}.
\end{prob}

\begin{prob}
Find a different condition from $p \equiv 2 \pmod{3}$ such that there exists a nonhyperelliptic superspecial curve of genus $4$ in characteristic $p$.
Cf.\ in the case of $g=1$, the elliptic curve $E : y^2 = x^3 + x$ is supersingular if $p \equiv 3 \pmod{4}$ and ordinary if $p \equiv 1 \pmod{4}$.
$($Also for hyperelliptic curves, such conditions are already found, see e.g., {\rm \cite{Taf}} and {\rm \cite{Taf2}}.$)$
\end{prob}

%======================================================================
\section{Application: Finding maximal curves over $K = \mathbb{F}_{p^2}$ for large $p$}\label{sec:comp}
%======================================================================
In the following, we set $K := \mathbb{F}_{p^2}$ with $p > 2$.
It is known that any maximal or minimal curve over $\mathbb{F}_{p^2}$ is supersepcial.
Conversely, any superspecial curve over an algebraically closed field descends to a maximal or minimal curve over $\mathbb{F}_{p^2}$, see the proof of \cite[Proposition 2.2.1]{KH16}.
Recall from Theorem \ref{thm:main} that $C_p = V(Q,P)$ with $Q = 2 y w + z^2$ and $P= x^3 + y^3 + w^3$ is a superspecial curve of genus $4$ if and only if $p \equiv 2 \pmod{3}$.
We computed the number of $\mathbb{F}_{p^2}$-rational points on $C_p$ for $3 \leq p \leq 269$ using a computer algebra system Magma \cite{Magma}.
Table \ref{table:3} shows our computational results for $3 \leq p \leq 100$.
We see from Table \ref{table:3} that any superspecial $C_p$ is maximal over $\mathbb{F}_{p^2}$ for $3 \leq p \leq 100$ (also for $101 \leq p \leq 269$, but omit to write them in the table).
From our computational results, let us give a conjecture on the existence of $\mathbb{F}_{p^2}$-maximal nonhyperelliptic curves of genus $4$.

\renewcommand{\arraystretch}{1.5}
\begin{table}[t]
\centering{
\caption{The number of $\mathbb{F}_{p^2}$-rational points on $C_p = V(Q,P)$ for $3 \leq p \leq 100$, where $Q = 2 y w + z^2$ and $P= x^3 + y^3 + w^3$.
We denote by $\#C_p (\mathbb{F}_{p^2})$ the number of $\mathbb{F}_{p^2}$-rational points on $C_p $ for each $p$.
}
\label{table:3}
\scalebox{0.97}{
\begin{tabular}{c||c|c|c||c||c|c|c} \hline
$p$ & $p \bmod{3}$ & S.sp.\ or not & $\# C_p ( \mathbb{F}_{p^2} )$  & $p$ & $p \bmod{3}$ & S.sp.\ or not & $\# C_p ( \mathbb{F}_{p^2} )$ \\ \hline
$3$ & $0$ & Not S.sp. & $10$  & $43$ & $1$ & Not S.sp. & $1938$   \\ \hline
$5$ & $2$ & S.sp. & $66$ (Max.)  & $47$ & $2$ & S.sp. & $2586$ (Max.)  \\ \hline
$7$ & $1$ & Not S.sp. & $48$             & $53$ & $2$ & S.sp. & $3234$ (Max.)  \\ \hline
$13$ & $1$ & Not S.sp. & $192$             &  $59$ & $2$ &  S.sp. & $3954$ (Max.) \\ \hline
$11$ & $2$ & S.sp. & $210$ (Max.)  &  $61$ & $1$ & Not S.sp. & $3648$ \\ \hline
$17$ & $2$ & S.sp. & $426$ (Max.)  &  $67$ & $1$ & Not S.sp. & $4368$ \\ \hline
$19$ & $1$ & Not S.sp. & $336$             &    $71$ & $2$ & S.sp.  &  $5610$ (Max.)    \\ \hline
$23$ & $2$ & S.sp. &  $714$ (Max.) &  $73$ & $1$ &  Not S.sp. & $5376$  \\ \hline
$29$ & $2$ & S.sp. & $1074$ (Max.)  & $79$ & $1$ & Not S.sp.  & $6384$ \\ \hline
$31$ & $1$ & Not S.sp. & $1146$                 &  $83$ & $2$ & S.sp.  & $7554$ (Max.)  \\ \hline
$37$ & $1$ & S.sp. & $1334$                        & $89$ & $2$ & S.sp.  & $8634$ (Max.) \\ \hline
$41$ & $2$ & S.sp. & $2010$ (Max.)  &  $97$ & $1$ & Not S.sp.  & $9408$ \\ \hline
\end{tabular}
}}
\end{table}

\begin{conje}
For any $p$ with $p \equiv 2 \pmod{3}$, the curve $C_p$ over $\mathbb{F}_{p^2}$ is maximal.
\end{conje}

\begin{rem}
We can reduce computing the number of $\mathbb{F}_{p^2}$-rational points on $C_p$ into computing that of zeros of a {\it diagonal} equation.
Specifically, by $2 y w + z^2 = 0$ and $x^3 + y^3 + w^3 = 0$, we have $x^3 + y^3 + (- z^2 / (2 y)^{-1})^3 = 0$ and thus $8 x^3 y^3 + 8 y^6 - z^6 = 0$ if $y \neq 0$.
Putting $X = x y$, one has the diagonal equation $8 X^3 + 8 y^6 - z^6 = 0$.
Hence, we may apply known methods to count the number of rational points of diagonal equations, see e.g., \cite{Wan} and \cite{Weil}.
At the time of this writing (as of April 24, 2018), however, we have not succeeded in applying any known method.
\end{rem}

%==================

\end{document}